\documentclass[10pt,a4paper]{article}
\pdfoutput=1

\usepackage[utf8]{inputenc} % allow utf-8 input
\usepackage[T1]{fontenc}    % use 8-bit T1 fonts
\usepackage{hyperref}       % hyperlinks
\usepackage{url}            % simple URL typesetting
\usepackage{booktabs}       % professional-quality tables
\usepackage{amsfonts}       % blackboard math symbols
\usepackage{nicefrac}       % compact symbols for 1/2, etc.
\usepackage{microtype}      % microtypography
\usepackage{bbm}
\usepackage{bm}
\usepackage{mathtools}
\usepackage[linesnumbered,algoruled,boxed,lined]{algorithm2e}
\usepackage{amsthm}
\usepackage{graphicx}
\usepackage{array}
\SetKwInOut{Parameter}{parameter}
\newtheorem{theorem}{Theorem}
\newtheorem{lemma}{Lemma}
\newtheorem{proposition}{Proposition}
\newtheorem{corollary}{Corollary}

\theoremstyle{definition}
\newtheorem*{example}{Example}

\theoremstyle{remark}

\newcommand{\z}{\zeta}
\newcommand{\R}{\mathbb{R}}
\newcommand{\OO}{\mathcal{O}}

\newcommand{\one}{\mathbbm{1}}
\mathchardef\mhyphen="2D % Define a "math hyphen"

\newcommand\restr[2]{{% we make the whole thing an ordinary symbol
  \left.\kern-\nulldelimiterspace % automatically resize the bar with \right
  #1 % the function
  \vphantom{\big|} % pretend it's a little taller at normal size
  \right|_{#2} % this is the delimiter
  }}

\title{Approximating the Convex Hull via Metric Space Magnitude}

\author{%
  Glenn Fung \\
  American Family Insurance\\
  Madison, WI 53783 \\
  \texttt{gfung@amfam.com} \\
  \and
  Eric Bunch \\
  American Family Insurance\\
  Madison, WI 53783 \\
  \texttt{ebunch@amfam.com} \\
  \and
  Dan Dickinson \\
  American Family Insurance\\
  Madison, WI 53783 \\
  \texttt{ddickins@amfam.com} \\
}

\begin{document}

\maketitle

\begin{abstract}
Magnitude of a finite metric space and the related notion of magnitude functions on metric spaces is an active area of research in algebraic topology.   Magnitude originally arose in the context of biology, where it represents the number of effective species in an environment; when applied to a one-parameter family of metric spaces $tX$ with scale parameter $t$, the magnitude captures much of the underlying geometry of the space. Prior work has mostly focussed on properties of magnitude in a global sense; in this paper we restrict the sets to finite subsets of Euclidean space and investigate its individual components. We give an explicit formula for the corrected inclusion-exclusion principle, and define a quantity associated with each point, called the \textit{moment} which gives an intrinsic ordering to the points. We exploit this in order to form an algorithm which approximates the convex hull. 
\end{abstract}

\section{Introduction}
The magnitude of a metric space is a construction that has recently garnered attention  \cite{Leinster2013TheMO}, \cite{LeinsterWillerton2013}, \cite{LeinsterMeckes16}, \cite{Barcelo18CptEuclid}. The intuition behind magnitude is often described as the ``effective number of points'' in a space. In this paper we posit a solution to the question ``which points?'' We use the magnitude of a metric space to define a quantity called the \textit{moment} associated with each point which we show captures relevant geometric information. In particular, we demonstrate how to use the moment of points to reduce the number of points needed when approximating the convex hull. We provide arguments that removing points according to Algorithm \ref{CHApprox} will not affect the magnitude of the set more than a pre-defined threshold. Further discussion suggests that the volume of the convex hull of a collection of points will not differ extremely from the volume of the convex hull of a subset of points, when the subset is chosen according to Algorithm \ref{CHApprox}. We discuss computational complexity of Algorithm \ref{CHApprox}, and discuss results of numerical experiments that approximate the convex hull of various collections of data points.

In previous work, properties of the magnitude operation $X \mapsto |X|$ have been studied, with a broad scope including enriched categories, non-Euclidean metric spaces, and infinite subsets of $\R^n$. In this paper we focus on finite subsets of $\R^n$, and in particular investigate more closely the importance of the weight vector $w = \z_X^{-1}\one$ for a finite set $X \subset \R^n$. In Section \ref{background} we give brief definitions, theorems and examples to set up for the sequel. In Section \ref{section_moment}, we investigate $|X\setminus Y|$ for finite sets $Y \subset X \subset \R^n$, and show in Lemma \ref{lemma_magnitude_restriction} that 

\[
|X\setminus Y| = |X| - \restr{w_X}{Y}^T\zeta_X/\zeta_{X\setminus Y}\restr{w_X}{Y}
\]

where $w_X$ is the weight vector for $X$, and $\restr{w_X}{Y}$ is the restriction of $w_X$ to the indices corresponding to $Y$, and $\zeta_X/\zeta_{X\setminus Y}$ is the Schur complement of $\zeta_{X \setminus Y}$ in $\zeta_X$. We use this to show in Proposition \ref{prop_magnitude_restriction_leq}

\[
|X| \geq |X \setminus Y| \geq |X| - N_Y\left( \max\{ w_X(y)^2 \mid y \in Y \} \right)
\]

where $N_Y$ is the number of points in $Y$. This shows that if $w_X$ is known, and the subset $Y$ can be chosen such that $w_X(y)$ is small for all $y \in Y$, then $|X \setminus Y|$ will be close to $|X|$. The informal discussion in \ref{discussion} suggests that by defining the moment of a point, denoted $\mu_0(x)$, we can condition on $\mu_0(y)$ instead of $w_X(y)$ to decide which subset $Y \subset X$ to remove, and the resulting set $X \setminus Y$ will be a fair approximation to $X$ in the sense that $Vol(Conv(X\setminus Y))$ will be close to $Vol(Conv(X))$. Here $Conv(X)$ denotes the convex hull of $X$.

In Section \ref{section_CH}, we formalize the process described above in Algorithm \ref{CHApprox}, and give a runtime analysis. We then run experiments using the algorithm to approximate the volume of the convex hull of data sets and display the results in Table \ref{ch_approx_exp}.

\section{Background}\label{background}

%Given magnitude's origins in biology and use in mathematics as the ``effective number'' of species or points respectively, a natural question is ``which ones?''

\subsection{Definitions}

Let $X \subset \R^n$ be a finite set of size $N$, i.e. $\#X = N$. Write $X = \{ x_1, x_2,...,x_N \}$. The \textbf{similarity matrix} of $X$ is $\z_X$, has entries $\z_X(i, j) = \exp(-\Vert x_i - x_j \Vert)$. In this setting, the inverse of $\z_X$ always exists; thus we can define the \textbf{weight vector} of $X$ to be $w = \z_X^{-1}\one$, and, following \cite{Leinster2013TheMO}, we define  the \textbf{magnitude of } $X$ to be $|X| := \sum_{i, j}\z_X^{-1}(i, j) = \sum_{i=1}^N w(x_i) = \one^T\z_X^{-1}\one$. Note that if we know $w$, then $|X| = w^T\z_Xw$.

As mentioned above, the focus of this paper centers on the importance of the elements of the weight vector. To this end, given $X = \{ x_1, x_2,...,x_N \} \subset \R^n$, define the \textbf{weight} of $x_i$, written $w(x_i)$, to be the corresponding entry in the weight vector, $w_i$. When the ordering of $X$ is understood, we simply write $w(x)$.

For an arbitrary subset (e.g. not necessarily finite) $X \subset \R^n$, define the magnitude to be
\[
|X| = \sup\{ |Y| \mid Y \subset X \text{ is finite }  \}
\]

For a set $X \subset \R^n$, and $t \in (0, \infty]$, one can define a metric space $tX = (X, td)$ as having the same points as $X$, but distance metric to be scaled by $t$; i.e. $(td)(x, y) = td(x, y)$. Note that since $X \subset \R^n$, the metric space $tX$ is equivalent to the metric space consisting of points whose coordinates are those of $X$ scaled by $t$ in each coordinate with the usual metric on $\R^n$. In this paper we will denote this space by $tX$ as well and will not disambiguate, in order for ease of exposition. The \textbf{magnitude function} of $X \subset \R^n$ is defined to be the function $t \rightarrow |tX|$ for $t \in (0, \infty]$.

\begin{example}
Instead of showing explicit computations of the magnitude of certain special metric spaces, which is done extensively in \cite{Leinster2013TheMO}, \cite{LeinsterMeckes16}, and \cite{LeinsterWillerton2013},  we will show plots of some finite subsets of $\R^2$ and color the points corresponding to their weighting. This will be suggestive and set the stage for the sequel. Figure \ref{weight_plots} has examples of four data sets where the color of a point represents $\log(1 + w(x))$.
\end{example}

\begin{figure}
  \includegraphics[scale=0.4]{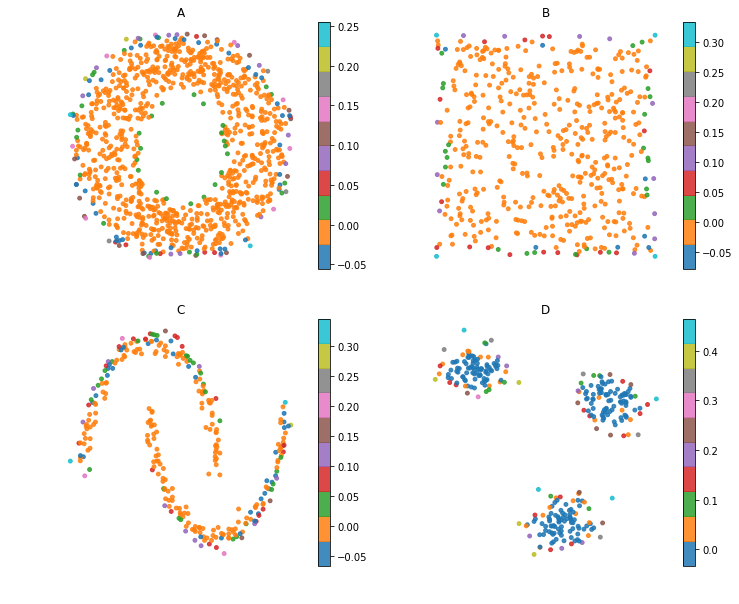}
  \caption{For each data set, the weight $w$ is calculated. The color of a point $x$ represents $\log(1+w(x))$. Synthetic data is plotted for A) Annulus B) Square C) Noisy moons D) Gaussian blobs.}
  \label{weight_plots}
\end{figure}

Figure \ref{weight_plots} hints that the weight of a point captures relevant geometric information. We formalize this for a three-point space in the following proposition.

\begin{proposition}\label{triangle_weight_prop}
Let $X \subset \R^n$ consist of three points, $X = \{x_1, x_2, x_3\}$, with similarity matrix $\z_X$ and weight vector $w$. If $\Vert x_1 - x_3 \Vert \geq \Vert x_2 - x_3 \Vert \geq \Vert x_1 - x_2 \Vert$, then $w(x_3) \geq w(x_2) \geq w(x_1)$.
\end{proposition}

%a = \Vert x_1 - x_2 \Vert \exp(-\Vert x_1 - x_2 \Vert)
%b = \Vert x_2 - x_3 \Vert \exp(-\Vert x_2 - x_3 \Vert)
%c = \Vert x_1 - x_3 \Vert \exp(-\Vert x_1 - x_3 \Vert)

\begin{proof}
Let $a = \Vert x_1 - x_2 \Vert$, $b = \Vert x_2 - x_3 \Vert$, and $c = \Vert x_1 - x_3 \Vert$, so $c \geq b \geq a > 0 $. Then solving $\z_X \dot w = \one$, we have \\
$w(x_1) = (-e^{-a}e^{-b} + e^{-a} + e^{-2b} - e^{-b}e^{-c} + e^{-c} - 1)/({-\det(\z_X)})$, \\
$w(x_2) = (-e^{-a}e^{-c} + e^{-a} + e^{-2c} - e^{-b}e^{-c} + e^{-b} - 1)/({-\det(\z_X)})$, \\
$w(x_3) = (-e^{-a}e^{-b} + e^{-b} + e^{-2a} - e^{-a}e^{-c} + e^{-c} - 1)/({-\det(\z_X)})$.\\
Then 
\begin{align}
w(x_1) - w(x_2)  &= e^{-a}e^{-b} + e^{-a}e^{-c} + e^{-2b} - e^{-b} - e^{-2c} + e^{-c}) / ({-\det(\z_X)}) \nonumber\\
&= -(e^{-b} - e^{-c})((e^{-a} - e^{-b}) + (1 - e^{-c})) / ({-\det(\z_X)}) \nonumber\\
&= (e^{-b} - e^{-c})((e^{-a} - e^{-b}) + (1 - e^{-c})) / {\det(\z_X)} \nonumber
\end{align}
Since $c \geq b \geq a > 0 $, $e^{-b} \geq e^{-c}$, $e^{-a} \geq e^{-b}$, and $1 \geq e^{-c}$, so $(e^{-b} - e^{-c})((e^{-a} - e^{-b}) + (1 - e^{-c})) \geq 0$. By Theorem \ref{theorem_pos_def}, $\det(\z_X) > 0$, so $(e^{-b} - e^{-c})((e^{-a} - e^{-b}) + (1 - e^{-c})) / {\det(\z_X)} \geq 0$ proving $w(x_1) \geq w(x_2)$.

Similarly, 
\begin{align} 
w(x_3) - w(x_2)  &= -e^{-a}e^{-b} + e^{-b}e^{-c} + e^{-2a} - e^{-a} - e^{-2c} + e^{-c}) / ({-\det(\z_X)}) \nonumber\\
&= -(e^{-a} - e^{-c})((e^{-b} - e^{-c}) + (1 - e^{-a})) / ({-\det(\z_X)}) \nonumber\\
&= (e^{-a} - e^{-c})((e^{-b} - e^{-c}) + (1 - e^{-a})) / {\det(\z_X)} \nonumber
\end{align}
Noting $e^{-a} \geq e^{-c}$, $e^{-b} \geq e^{-c}$, and $1 \geq e^{-a}$, so $(e^{-a} - e^{-c})((e^{-b} - e^{-c}) + (1 - e^{-a})) \geq 0$. By Theorem \ref{theorem_pos_def}, $\det(\z_X) > 0$, so $(e^{-a} - e^{-c})((e^{-b} - e^{-c}) + (1 - e^{-a}))  / {\det(\z_X)} \geq 0$ proving $w(x_3) \geq w(x_3)$.

Finally,
\begin{align} 
w(x_3) - w(x_1)  &= -e^{-a}e^{-c} + e^{-b}e^{-c} + e^{-2a} - e^{-a} - e^{-2b} + e^{-b}) / ({-\det(\z_X)}) \nonumber\\
&= (e^{-a} - e^{-b})(e^{-a} -1 + e^{-b} - e^{-c}) / ({-\det(\z_X)}) \label{w3w1} 
\end{align}
Let $f(x) = e^{-x}$, with $g_z(x) = e^{-z}+ze^{-z} - xe^{-z}$ denoting the tangent line of $f(x)$ at $z$. The convexity of $f$ implies $f(x) \geq g_z(x)$ for any $z$ and $x$. Taking $z=b$ and $x=c$, \\
$e^{-b} +be^{-b} - ce^{-b} \leq e^{-c}$, so 
\begin{align}
e^{-b} - e^{-c} &\leq e^{-b}(c - b) \nonumber\\
&\leq e^{-a}(c - b) \label{abc}\\
&\leq ae^{-a} \label{abctri}
\end{align}
using $e^{-b} \leq e^{-a}$ on line \ref{abc} the triangle inequality on line \ref{abctri}. 
Thus \begin{align} (e^{-a} - 1 + e^{-b} - e^{-c}) &\leq (e^{-a} - 1 + ae^{-a}) \nonumber\\
&\leq 0\nonumber
\end{align}
Where the last step has used convexity of $f$, with $z=a$ and $x=0$ to yield $e^{-a}(1+a) \leq 1$.
The numerator of line \ref{w3w1} is therefore non-positive, and since the denominator is negative by Theorem \ref{theorem_pos_def} $w(x_3) \geq w(x_1)$.

\end{proof}

\subsection{Properties and Theorems}

In this section we will collect and record some results that will be useful in the sequel. The following three important results ensure that for finite sets $X \subset \R^n$ the similarity matrix and magnitude function are reasonably well behaved (Theorem \ref{theorem_t_inf_finite_set} actually holds for arbitrary finite metric spaces).

\begin{theorem}[Theorem 2.5.3, \cite{Leinster2013TheMO}]\label{theorem_pos_def}
$\z_X$ and thus $\z_X^{-1}$ are symmetric positive definite for finite sets $X \subseteq \R^n$.
\end{theorem}

\begin{theorem}[Proposition 2.2.6 \cite{Leinster2013TheMO}]\label{theorem_analytic}
The magnitude function of $X \subseteq \R^n$ is analytic at all $t \in (0, \infty]$.
\end{theorem}

\begin{theorem}[Proposition 2.2.6 \cite{Leinster2013TheMO}]\label{theorem_t_inf_finite_set}
For $X \subset \R^n$ finite, $\lim_{t \rightarrow \infty} |tX| = N$, where $N$ is the number of points in $X$.
\end{theorem}

Theorem \ref{theorem_pos_def} will be used extensively in the sequel, and Theorem \ref{theorem_analytic} is used implicitly in the definition of the moment of a point in Section \ref{section_moment}. Next, if we concern ourselves with a finite set $X \subset \R^n$, and a subset $Y \subset X$, then the following theorem gives a nice relationship between $|Y|$ and $|X|$.

\begin{theorem}[Corollary 2.10 \cite{LeinsterMeckes16}]\label{theorem_subset_mag_leq}
For $X, Y \subset \R^n$ finite with $\emptyset \neq Y \subseteq X$, then $1 \leq |Y| \leq |X|$. 
\end{theorem}

The next theorem gives a useful way to approximate the magnitude of an infinite set in $\R^n$

\begin{theorem}[Corollary 2.7 \cite{MeckesPosDef}]\label{theorem_finite_hd_approx}
If $X \subset \R^n$ is compact and $\{ X_k \}$ is a sequence of finite subsets of $\R^n$ such that $\lim_{k \rightarrow \infty} d_H(X_k, X) = 0$, then $|X| = \lim_{k \rightarrow \infty}X_k$ where $d_H$ is the Hausdorff distance.
\end{theorem}

The following theorem gives a concrete connection between the magnitude function of an infinite subset $X \subset \R^n$ and the volume of that subset. This will be used in the informal discussion in Section \ref{discussion}.

\begin{theorem}[Theorem 1 \cite{Barcelo18CptEuclid}]\label{theorem_vol}
For $X \subset \R^n$ nonempty and compact, we have

\[
\lim_{t \rightarrow 0^+} |tX| = 1
\]

and 

\[
\lim_{t \rightarrow \infty}\frac{|tX|}{t^n} = \frac{Vol(X)}{n!Vol(B_n)}
\]

where $B_n \subset \R^n$ is the unit ball.

\end{theorem}

\section{Zeroth Moment of a Point}\label{section_moment}

\subsection{Definition}

Figure \ref{weight_plots} indicates that the weight of a point is informative and potentially useful in analyzing a data set. However, it is desirable to have a version of the weight of a point that does not depend on $t$. To that end, we have the following definition. For a finite set $X \subset \R^n$, denote by $w_t$ the weight vector for $tX$. Define the \textbf{1-shifted power zeroth moment} of $x \in X$ to be

\begin{equation}\label{moment}
\mu_0(x) = \int_0^{\infty} e^{-t}w_t(x)^2dt
\end{equation}

We will also refer to $\mu_0(x)$ as the \textbf{moment} of $x$. 

\begin{example}
Figure \ref{moment_plots} shows examples of data sets in $\R^2$ where the color of a point $x$ represents $\log(1 + \mu_0(x))$.
\end{example}

%\begin{remark}
%Although not investigated in this paper, the definition above suggests that the quantities
%
%\[
%\mu_n(x) =  \int_0^{\infty} t^n e^{-t}w_t(x)^2dt
%\]
%
%are interesting, as well as the (shifted) Laplace transform of $w_t(x)^2$:
%
%\[
%\mathcal{L}\{w_t\}(s) = \int_0^{\infty}  e^{-(s+1)t}w_t(x)^2dt
%\]
%\end{remark}

\begin{figure}
  \includegraphics[scale=0.4]{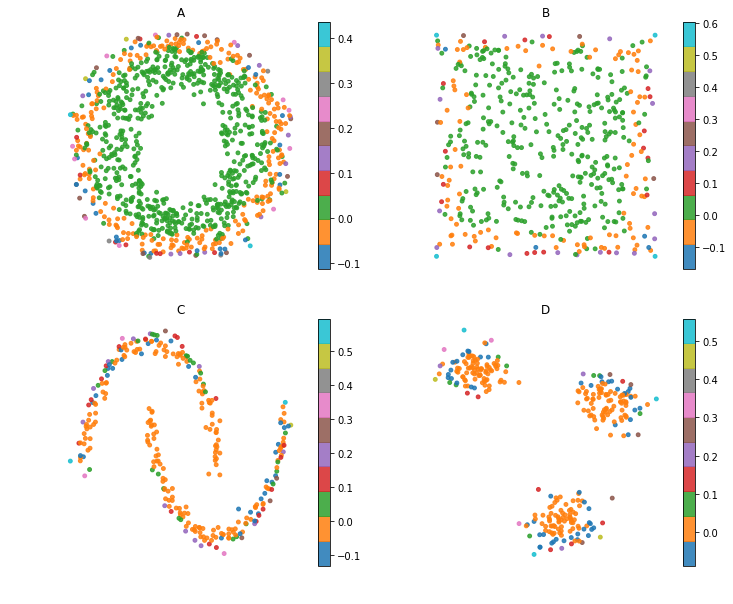}
  \caption{For each data set, the zeroth moment $\mu_0(x)$ is calculated. The color of a point $x$ represents $\log(1+\mu_0(x))$. Synthetic data is plotted for A) Annulus B) Square C) Noisy moons D) Gaussian blobs.}
  \label{moment_plots}
\end{figure}

As with weight, the moment of a point captures relevant geometric information in a three point set.

\begin{corollary}\label{triangle_weight_cor}
Let $X \subset \R^n$ consist of three points, $X = \{x_1, x_2, x_3\}$. If $\Vert x_1 - x_3 \Vert \geq \Vert x_2 - x_3 \Vert \geq \Vert x_1 - x_2 \Vert$, then $\mu_0(x_3) \geq \mu_0(x_2) \geq \mu_0(x_1)$.
\end{corollary}

\begin{proof}
By [Proposition 2.4.15 \cite{Leinster2013TheMO}], $w(x_i) > 0$ for all $i$, and by Proposition \ref{triangle_weight_prop} $w(x_3) \geq w(x_2)$. For any value of $t \in (0,\infty)$, $tX$ is simply another three point space with the same inequalities between the $x_i$, we have $w_t(x_3) \geq w_t(x_2)$ for all $t$, and therefore $w_t(x_3)^2 \geq w_t(x_2)^2$, where $w_t$ is the weight vector of $tX$.
Thus, $\mu_0(x_3) - \mu_0(x_2) = \int_0^{\infty} e^{-t}w_t(x_3)^2dt - \int_0^{\infty} e^{-t}w_t(x_3)^2dt = \int_0^{\infty} e^{-t}(w_t(x_3)^2 - w_t(x_2)^2)dt \geq 0$. The proof is similar for other pairings.
\end{proof}

\subsection{Computation}

Let $P \subset \{1,..., N\}$. Define $X_P :=  \{ x_p \mid p \in P \} \subseteq X$. Without loss of generality, assume $P$ consists of the first $l$ elements, where $l = |P|$, the number of elements in $P$. Denote by $\bar{P}$ the complement of $P$ in $\{ 1,...,N \}$. 

For an $n\times m$ matrix $M$, and $P \subseteq \{ 1,...,n  \}$, $Q \subseteq \{ 1,...,m  \}$, denote by $M_{PQ}$ the submatrix of $M$ obtained by removing all rows whose index is in $\bar{P}$, and all columns whose index is in $\bar{Q}$. If $P = Q$, then we will write $M_P = M_{PP}$.

For simplicity, set $A = \z_X$.  Then if we write $A$ as a the block matrix

\[
A = \begin{bmatrix}
A_P & A_{P \bar{P}} \\
A_{P \bar{P}}^T & A_{\bar{P}}
\end{bmatrix}
\]

we can write the formula $Aw = \one$ as the following system of equations

\begin{align*}
A_Pw[P] + A_{P \bar{P}}w[\bar{P}] &= \one_P \\
A_{P \bar{P}}^T w[P] + A_{\bar{P}}w[\bar{P}] &= \one_{\bar{P}}
\end{align*}

where $\one_P$ is the $|P|\times 1$ column vector whose entries are all $1$, and $w[P]$ is the $|P| \times 1$ column vector obtained by taking only the rows of $w$ whose index is in $P$; i.e. it is the column vector $[w(x_{p_1}) ... w(x_{p_l}))]^T$. Since $A_{\bar{P}}$ is invertible, we can solve for $w[P]$:

\begin{equation}\label{equation_magnitude_union_pt1}
w[P] = (A/A_{\bar{P}})^{-1}(\one_P - A_{P\bar{P}}w_{\bar{P}})
\end{equation}

where $A/A_{\bar{P}} = A_P - A_{P \bar{P}}A_{\bar{P}}^{-1}A_{P \bar{P}}^T$ is the \textbf{Schur complement} of $A_{\bar{P}}$ in $A$,  and $w_{\bar{P}}$ is the weight vector for $X_{\bar{P}} $. That is, $w_{\bar{P}} = A_{\bar{P}}^{-1}\one_{\bar{P}}$. In a similar fashion, since $A_P$ is invertible, we can solve for $w[\bar{P}]$:

\begin{equation}\label{equation_magnitude_union_pt2}
w[\bar{P}] = (A/A_P)^{-1}(\one_{\bar{P}} - A_{P\bar{P}}^Tw_P)
\end{equation}

Thus if we have finite sets $X, Y \subset \R^n$ such that $X \cap Y = \emptyset$, we can calculate the weighting of $X \cup Y$ given that we know the weightings of $X$ and $Y$ individually, along with the matrix $\z_{X \cup Y}$.

Since $A_{\bar{P}P}$ and $A_{P\bar{P}}^T$ are not invertible, we cannot use the above to calculate $w_P$ or $w_{\bar{P}}$ in terms of $w$. So we wish to calculate $|X_{\bar{P}}|$ in terms of the weight vector $w$ on the entire space $X$. 

\begin{lemma}\label{lemma_magnitude_restriction}
For finite sets $Y \subset X \subset \R^n$, we have that 
	\begin{equation}\label{eqn_magnitude_restriction}
	|X \setminus Y| = |X| - \restr{w_X}{Y}^T\zeta_X/\zeta_{X\setminus Y}\restr{w_X}{Y}
	\end{equation}
where $w_X$ is the weight vector of $X$,  $\restr{w_X}{Y}$ is $w_X$ restricted to the indices corresponding to $Y$, and $\zeta_X/\zeta_{X\setminus Y}$ is the Schur complement of $\z_{X \setminus Y}$ in $\z_X$.
\end{lemma}

\begin{proof}
For simplicity, set  $B = A^{-1} = \z_X^{-1}$, and denote by $a_{ij}$ and $b_{ij}$ the elements of $A$ and $B$ respectively. We can see that $|X_{\bar{P}}| = \one^TA_{\bar{P}}^{-1}\one$. Now for $i, j \in \bar{P}$, we use Theorem 5.1 from \cite{Ruiz2016RelationshipBT} to write

\begin{align*}
A_{\bar{P}}^{-1}(i, j) 
&= 
\frac{det\left(
\begin{bmatrix}
b_{p_1p_1} & b_{p_1p_2} & \cdots & b_{p_1p_l} & b_{p_1j} \\
b_{p_2p_1} & \cdots & & & b_{p_2j} \\
\vdots & & & & \vdots \\
b_{p_lp_1} & & & & b_{p_lj} \\
b_{ip_1} & b_{ip_2} & \cdots & b_{ip_l} & b_{ij}
\end{bmatrix}
\right)}{det \left( 
\begin{bmatrix}
b_{p_1p_1} & b_{p_1p_2} & \cdots & b_{p_1p_l}  \\
b_{p_2p_1} & \cdots & &  \\
\vdots & & &  \\
b_{p_lp_1} & & & b_{p_lp_l} \\
\end{bmatrix}
\right)}\\ 
&= 
\frac{1}{det(B_P)}det \left( 
\begin{bmatrix}
B_P & B_{Pj} \\
B_{Pj}^T & b_{ij}
\end{bmatrix} 
\right)
\end{align*}

where $P = \{ p_1,...,p_l \}$ with $p_1 < p_2 < \cdots p_l$, and  $B_{Pj}$ is the $l \times 1$ column vector $[b_{p_1j} b_{p_2j} ... b{p_lj}]^T$. Note that since $A$ is symmetric, $B$ is also symmetric, and thus $b_{ij} = b_{ji}$. Now calculate the $i^{th}$ row sum of $A_{\bar{P}}^{-1}$:

\begin{align*}
A_{\bar{P}}^{-1}\one_{\bar{P}}[i] 
&= 
\sum_{j \notin P} A_{\bar{P}^{-1}}(i, j) 
= 
\frac{1}{det(B_P)}
\sum_{j \notin P}det \left( 
\begin{bmatrix}
B_P & B_{Pj} \\
B_{Pi}^T & b_{ij} \\
\end{bmatrix} 
\right) \\
&= 
\frac{1}{det(B_P)} 
det \left( 
\begin{bmatrix}
B_P & \sum_{j \notin P} B_{Pj} \\
B_{Pi}^T & \sum_{j \notin P} b_{ij}
\end{bmatrix} 
\right) \\
&= 
\frac{1}{det(B_P)} det \left( 
\begin{bmatrix}
B_P & \sum_{j=1}^N B_{Pj} - \sum_{\alpha \in P}B_{P\alpha} \\
B_{Pi}^T & \sum_{j=1}^N b_{ij} - \sum_{\alpha \in P} b_{i \alpha}
\end{bmatrix}  
\right) \\
&= 
\frac{1}{det(B_P)} det \left( 
\begin{bmatrix}
B_P & w[P] \\
B_{Pi}^T & w(x_i)
\end{bmatrix} 
\right) - 
\frac{1}{det(B_P)} det \left( 
\begin{bmatrix}
B_P & \sum_{\alpha \in P} B_{P \alpha} \\
B_{Pi}^T & \sum_{\alpha \in P} b_{i \alpha}
\end{bmatrix}  
\right) \\
&= 
\frac{1}{det(B_P)} det \left( 
\begin{bmatrix}
B_P & w[P] \\
B_{Pi}^T & w(x_i) 
\end{bmatrix} 
\right)
=
\frac{det(B_P)}{det(B_P)}( w(x_i) - B_{Pi}^TB_P^{-1}w[P] ) \\
&= 
w(x_i) - B_{Pi}^TB_P^{-1}w[P]
\end{align*}

Now we sum over $i$ to find the sum of all the elements in $A_{\bar{P}}^{-1}$ to calculate $|X_{\bar{P}}|$: 

\begin{align*}
|X_{\bar{P}}| 
&= 
\one_{\bar{P}}^TA_{\bar{P}}^{-1}\one _{\bar{P}}
= 
\sum_{i \notin P} A_{\bar{P}}^{-1}\one_{\bar{P}}[i] 
= 
\sum_{i \notin P} w(x_i) - B_{Pi}^TB_P^{-1}w[P] \\
&= 
\sum_{i}w(x_i) - \sum_{\alpha \in P} w(x_{\alpha}) - \sum_{i} B_{Pi}^TB_P^{-1}w[P] + \sum_{\alpha \in P} B_{P\alpha}^TB_P^{-1}w[P] \\
&= 
|X| - \sum_{\alpha \in P} w(x_{\alpha}) - \sum_{i} B_{Pi}^TB_P^{-1}w[P] + \sum_{\alpha \in P}w(x_{\alpha}) \\
&= 
|X| - \sum_{i} B_{Pi}^TB_P^{-1}w[P] = |X| - w[P]^TB_P^{-1}w[P]
\end{align*}

Now since $A$ is positive definite, $B$ is positive definite, thus the principal submatrix $B_P$ is positive definite, hence $B_P^{-1}$ is positive definite. We now calculate $B_P^{-1} = B_{\bar{\bar{P}}}^{-1}$. For $i, j \in P$:

\begin{align*}
B_P^{-1}(i, j)
=
B_{\bar{\bar{P}}}^{-1}(i, j) 
&= 
\frac{1}{det(A_{\bar{P}})} det \left( 
\begin{bmatrix}
A_{\bar{P}} & A_{\bar{P}j} \\
A_{\bar{P}i}^T & a_{ij}
\end{bmatrix}  
\right) 
= 
\frac{det(A_{\bar{P}})}{det(A_{\bar{P}})}(a_{ij} - 
A_{\bar{P}i}^T A_{\bar{P}}^{-1} A_{\bar{P}j}) \\
&= 
a_{ij} - A_{\bar{P}i}^T A_{\bar{P}}^{-1} A_{\bar{P}j}
\end{align*}

Thus

\[
B_P^{-1} = B_{\bar{\bar{P}}}^{-1} = A_P - A_{\bar{P}P}^TA_{\bar{P}}^{-1}A_{\bar{P}P}
\]

Since $A$ is symmetric, we have that $A_{\bar{P}P}^T = A_{P\bar{P}}$, and

\[
B_P^{-1} = A_P - A_{P \bar{P}}A_{\bar{P}}^{-1}A_{P \bar{P}}^T
\]

If we write $A$ as the block matrix 

\[
A 
= 
\begin{bmatrix}
A_P & A_{P \bar{P}} \\
A_{P \bar{P}}^T & A_{\bar{P}}
\end{bmatrix}
\]

we can see that $B_P^{-1}$ is the Schur complement of $A_{\bar{P}}$ in $A$, denoted $A/A_{\bar{P}}$. Thus we can write

\[
|X_{\bar{P}}| = |X| - w[P]^TA/A_{\bar{P}}w[P]
\]

\end{proof}

We know that $det(A/A_{\bar{P}}) = \frac{det(A)}{det(A_{\bar{P}})}$. Fischer's inequality gives that $det(A) \leq det(A_P)det(A_{\bar{P}})$. Since $A_P$ is positive definite, and has all ones on the diagonal, Hadamard's inequality gives that $det(A_P) \leq 1$. Thus we can see that $det(A/A_{\bar{P}}) = \frac{det(A)}{det(A_{\bar{P}})} \leq det(A_P) \leq 1$.

\begin{proposition}\label{prop_magnitude_restriction_leq}
For finite sets $Y \subset X \subset \R^n$, we have that 

	\begin{align}\label{equation_magnitude_restr_ineq}
	|X| 
	&\geq 
	|X| - N_Ydet(\z_X/\z_{X \setminus Y})\left(  \min \{ w_X(y)^2 \mid y \in Y \} \right) \\
	& \geq
	|X \setminus Y| \nonumber \\
	& \geq 
	|X| - N_Y\left( \max\{ w_X(y)^2 \mid y \in Y \} \right) \nonumber \\ \nonumber
	\end{align}
where $w_X$ is the weight vector of $X$, and $N_Y$ is the number of points in $Y$.
\end{proposition}

\begin{proof}
Since $A/A_{\bar{P}}$ is positive definite, we know that $tr(A/A_{\bar{P}}) > 0$. Also, $tr(A_P) = |P|$ since $A_P$ has ones on its diagonal. Now

\[
tr(A/A_{\bar{P}}) 
= 
\sum_{i \in P} a_{ii} - A_{\bar{P}i}^TA_{\bar{P}}^{-1}A_{\bar{P}i} 
= 
\sum_{i \in P} 1 - A_{\bar{P}i}^TA_{\bar{P}}^{-1}A_{\bar{P}i} 
= 
|P| - \sum_{i \in P} A_{\bar{P}i}^TA_{\bar{P}}^{-1}A_{\bar{P}i} 
> 
0
\]

Since $A_{\bar{P}}^{-1}$ is positive definite, we know that $x^TA_{\bar{P}}^{-1}x > 0$ for all $x \in \R^{|\bar{P}|}\setminus \mathbf{0}$; in particular $A_{\bar{P}i}^TA_{\bar{P}}^{-1}A_{\bar{P}i} > 0$. Thus 

\[
tr(A_P) 
= 
|P| 
> 
|P| - \sum_{i \in P} A_{\bar{P}i}^TA_{\bar{P}}^{-1}A_{\bar{P}i} 
= 
tr(A/A_{\bar{P}}) 
> 
0
\]

Next, we know that

\[
w[P]^TA/A_{\bar{P}}w[P] 
= 
\sum_{i \in P} \lambda_i w(x_{i})^2
\]

where $\{ \lambda_i \mid i \in P \}$ are the eigenvalues of $A/A_{\bar{P}}$. We know that since $A/A_{\bar{P}}$ is positive definite, $\lambda_i > 0$ for all $i \in P$. Now let $\beta \in P$ be such that $w(x_{i})^2 \leq w(x_{\beta})^2$ for all $i \in P$. Then

\[
0 
< 
\sum_{i \in P}\lambda_i w(x_{i})^2 
\leq 
\sum_{i \in P}\lambda_i w(x_{\beta})^2 
= 
w(x_{\beta})^2 \sum_{i \in P} \lambda_i 
\leq 
|P|  w(x_{\beta})^2
\]

Thus we have that 

\begin{equation}\label{equation_magnitude_lb_ineq}
|X_{\bar{P}}| 
= 
|X| - w[P]^TA/A_{\bar{P}}w[P] 
\geq 
|X| - |P|  w(x_{\beta})^2
\end{equation}

By the inequality of arithmetic and geometric means, we have that 

\begin{align*}
\sum_{i \in P}\lambda_i w(x_{i})^2 
&\geq 
|P| \sqrt[\leftroot{-3}\uproot{3}|P|]{\prod_{i \in P}\lambda_i w(x_{i})^2} 
= 
|P|\sqrt[\leftroot{-3}\uproot{3}|P|]{det(A/A_{\bar{P}})\prod_{i \in P} w(x_{i})^2} \\
&\geq 
|P| det(A/A_{\bar{P}}) \sqrt[\leftroot{-3}\uproot{3}|P|]{\prod_{i \in P} w(x_{i})^2} 
= 
|P| det(A/A_{\bar{P}}) \prod_{i \in P} w(x_{i})^{\frac{2}{|P|}}
\geq 
0
\end{align*}

since $det(A/A_{\bar{P}}) \leq 1$. Let $\gamma \in P$ be such that $w(x_{\gamma})^2 \leq w(x_{i})^2$ for all $i \in P$. Then

\[
\sum_{i \in P}\lambda_i w(x_{i})^2 
\geq 
|P| det(A/A_{\bar{P}}) \prod_{i \in P} w(x_{i})^{\frac{2}{|P|}} 
\geq 
|P| det(A/A_{\bar{P}}) w(x_{\gamma})^2
\]

Then we have that

\begin{align*}
|X| 
&\geq 
|X| - |P| det(A/A_{\bar{P}}) w(x_{\gamma})^2 \\
&\geq 
|X| - w[P]^TA/A_{\bar{P}}w[P] \\
&= 
|X_{\bar{P}}| \\
&\geq 
|X| - |P|  w(x_{\beta})^2 
\end{align*}

\end{proof}

If we substitute $tX$ for $X$, $tX_{\bar{P}}$ for $X_{\bar{P}}$, and $w_t$ for $w$ in the above equation and take $t \rightarrow \infty$, we can see this reduces to what we expect. For $|tX| \rightarrow \#X = N$, $|tX_{\bar{P}}| \rightarrow N -|P|$, and $w_t(x) \rightarrow 1$ for all $x \in tX$. 

If we can choose $P$ such that $w(x_{i})$ is small for all $i \in P$, then removing all $x_{i}$ will not affect the magnitude that much; that is, $|X_{\bar{P}}|$ will be close to $|X|$. Explicitly, if $w(x_{\beta})$ can be chosen such that $w(x_{\beta}) < \frac{\epsilon}{|P|}$ for some $\epsilon > 0$, then $|X| \geq |X_{\bar{P}}| \geq |X| - \epsilon$.

We can now investigate the situation where $Z = X \cup Y$ where $X, Y \subset \R^n$ are finite, but are not necessarily disjoint. In order to calculate $|Z|$ given that we know $w_X$ and $w_Y$, we can combine Lemma \ref{lemma_magnitude_restriction} with Equations \ref{equation_magnitude_union_pt1} and \ref{equation_magnitude_union_pt2}. For completeness, we summarize this here

\begin{corollary}
Let $Z = X \cup Y$ where $X, Y \subset \R^n$ are finite sets. Set $W = X \setminus (X \cap Y) = X \setminus Y$. Then the weight vector $w_Z$ can be computed using the formulas

\begin{align} 
\restr{w_Z}{W} 
&= 
(\z_Z/\z_Y)^{-1} \left( \one_{W} - \z_Z[W][Y]w_Y \right) \\
\restr{w_Z}{Y}
&=
(\z_Z/\z_{W})^{-1} \left( \one_{Y} - \z^T_Z[W][Y][w_X - \z_X^{-1}[X\cap Y][W]^T \z_X/\z_{W} \restr{w_X}{X\cap Y}  \right) \\ 
\nonumber
\end{align}
where e.g. $\z_Z[W][Y]$ is the $N_W \times N_Y$ matrix corresponding to deleting all rows of $\z_Z$ except those whose index corresponds to points in $W$, and deleting all columns of $\z_Z$ except those whose index corresponds to points in $Y$.
\end{corollary}

\subsection{Discussion}\label{discussion}

Let us rewrite Equation \ref{equation_magnitude_lb_ineq} from Proposition \ref{prop_magnitude_restriction_leq} to include $t$:

\begin{equation}\label{mag_ineq_with_t}
|tX| \geq |tX_{\bar{P}}| \geq |tX| - |P|w_t(x_{\beta_t})^2
\end{equation}

One difficulty here is that the point $x_{\beta_t}$ varies with $t$. However, note that for some fixed 
$x \in X_{P}$, we have that $w_t(x_{\beta_t})^2 \geq w_t(x)^2$ for all $t \in (0, \infty)$. Thus we have that 

\[
|tX| \geq |tX| - |P|w_t(x)^2 \geq |tX| - |P|w_t(x_{\beta_{t}})^2
\]

but we do not know whether or not

\[
|tX_{\bar{P}}| \geq |tX| - |P|w_t(x)^2
\]

We can now see that if we integrate these quantities over $(0, \infty)$ against $e^{-t}$, we get that 

%\[
%\int_0^{\infty} e^{-t}|tX| dt 
%\geq 
%\int_0^{\infty} e^{-t} |tX_{\bar{P}}|dt 
%\geq 
%\int_0^{\infty} e^{-t}|tX| dt  - |P|\int_0^{\infty} e^{-t}w_t(x_{\beta_{t}})^2dt
%\]

\[
\mu_0(|tX|)
\geq
\mu_0(|tX_{\bar{P}}|)
\geq
 \mu_0(|tX|) -  |P|\int_0^{\infty} e^{-t}w_t(x_{\beta_{t}})^2dt
\]

where we will abuse notation and write $\mu_0(|tX|) = \int_0^{\infty} e^{-t}|tX| dt $. We also have that 

%\[
%\int_0^{\infty} e^{-t}|tX| dt  - |P|\mu_0(x) 
%\geq 
%\int_0^{\infty} e^{-t}|tX| dt  - |P|\int_0^{\infty} e^{-t}w_t(x_{\beta_{t}})^2dt
%\]

\[
\mu_0(|tX|) -  |P|\mu_0(x) 
\geq
\mu_0(|tX|) - |P|\int_0^{\infty} e^{-t}w_t(x_{\beta_{t}})^2dt
\]

Thus if we let $\hat{x}$ be the point in $X_P$ such that $\mu_0(\hat{x}) \geq \mu_0(x)$ for all $x \in X_P$, we can see that 

%\[
%\int_0^{\infty} e^{-t}|tX| dt  - |P|\mu_0(x) 
%\geq
%\int_0^{\infty} e^{-t}|tX| dt  - |P|\mu_0(\hat{x}) 
%\geq 
%\int_0^{\infty} e^{-t}|tX| dt  - |P|\int_0^{\infty} e^{-t}w_t(x_{\beta_{t}})^2dt
%\]

\[
\mu_0(|tX|) -  |P|\mu_0(x) 
\geq
\mu_0(|tX|)  - |P|\mu_0(\hat{x}) 
\geq
\mu_0(|tX|)  - |P|\int_0^{\infty} e^{-t}w_t(x_{\beta_{t}})^2dt
\]

We will in practice use the value $\mu_0(\hat{x})$ in place of $\int_0^{\infty} e^{-t}w_t(x_{\beta_{t}})^2dt$.

If we have a finite set $X \subset \R^n$, we can form its convex hull; denoted $Conv(X)$. Since $Conv(X)$ is an infinite subset of $\R^n$, it's magnitude can be defined as the supremum of magnitudes of all the finite subsets of it:

\[
|Conv(X)| = \sup\{ |Y| \mid Y \subseteq Conv(X)  \}
\]

From theorems \ref{theorem_subset_mag_leq} and \ref{theorem_finite_hd_approx} we can see that if $\{ X_k \}$ is a sequence of finite subsets of $Conv(X)$, each obtained by taking $k$ samples from a uniform distribution over the region of $Conv(X)$, we have that the sequence $|X_k|$ monotonically increases to $|Conv(X)|$. So if we assume that we are in the situation where the original set $X$ is close to (in Hausdorff distance) to one of the $X_k$ for $k$ large, then we can think of $X$ as being a discrete approximation for the infinite region $Conv(X)$.

%We know that if $X \subseteq Y \subset \R^n$ are finite sets, then $|X| \leq |Y|$. We also know that if $Z_k$ is a sequence of subsets of $\R^n$ such that $d_H(Z_k, Z) \rightarrow 0$ as $k \rightarrow \infty$ then $|Z_k| \rightarrow |Z|$ as $k \rightarrow \infty$. Here $d_H$ is the Hausdorff distance. Thus since $|Conv(X)|$ is finite, and we have a sequence of finite sets $X_k \subset Conv(X)$ such that $d_H(X_k, Conv(X)) \rightarrow 0$, then we know the sequence $|X_k|$ increases monotonically as $k$ increases, and $|X_k|$ converges to $|Conv(X)|$. Thus let us assume $X$ is close to $Conv(X)$ with respect to the Hausdorff distance. 

Although the function $|tX| \rightarrow N_X$ as $t \rightarrow \infty$, we will still think of $|tX|$ as an approximation to $|tConv(X)|$ for this discussion. We know that by Theorem \ref{theorem_vol}

\[
\lim_{t \rightarrow \infty}\frac{|tConv(X)|}{t^n} = \frac{Vol(Conv(X))}{n! Vol(B_n)}
\]

where $B_n$ is the unit ball in $\R^n$. So let us approximate $|tX|$ with the function $\frac{Vol(Conv(X))t^n}{n! Vol(B_n)}$. So the inequality \ref{mag_ineq_with_t} will be approximated with

%\[
%|tX| \geq |tX_{\bar{P}}| \geq |tX| - |P|w_t(x_{\beta_t})^2
%\]

\[
\frac{Vol(Conv(X))t^n}{n! Vol(B_n)} 
\geq 
\frac{Vol(Conv(X_{\bar{P}}))t^n}{n! Vol(B_n)} 
\geq 
\frac{Vol(Conv(X))t^n}{n! Vol(B_n)} - |P|w_t(x_{\beta_t})^2
\]

If we integrate this over $(0, \infty)$ against $e^{-t}$, we get

\begin{align*}
\int_0^{\infty} e^{-t} \frac{Vol(Conv(X))t^n}{n! Vol(B_n)} dt
& \geq 
\int_0^{\infty} e^{-t} \frac{Vol(Conv(X_{\bar{P}}))t^n}{n! Vol(B_n)} dt \\
&\geq 
\int_0^{\infty} e^{-t}  \frac{Vol(Conv(X))t^n}{n! Vol(B_n)} dt - 
|P|  \int_0^{\infty} e^{-t} w_t(x_{\beta_t})^2 dt
\end{align*}

This integrates out to

\begin{align*}
\frac{Vol(Conv(X))}{n Vol(B_n)} 
&\geq \frac{Vol(Conv(X_{\bar{P}}))}{n Vol(B_n)} \\
&\geq \frac{Vol(Conv(X))}{n Vol(B_n)} - |P|  \int_0^{\infty} e^{-t} w_t(x_{\beta_t})^2 dt
\end{align*}

Multiply through by $nVol(B_n)$ to get

\begin{align*}
Vol(Conv(X))
&\geq 
Vol(Conv(X_{\bar{P}})) \\
&\geq 
Vol(Conv(X)) - |P|nVol(Conv(X))\int_0^{\infty} e^{-t} w_t(x_{\beta_t})^2 dt
\end{align*}

Thus if the set we are removing whose indices correspond to $P$ can be chosen such that $\int_0^{\infty} e^{-t} w_t(x_{\beta_t})^2 dt \leq \frac{\epsilon}{|P|nVol(Conv(X))}$ for some $\epsilon > 0$, we have that 

\begin{align*}
Vol(Conv(X)) \geq Vol(Conv(X_{\bar{P}})) \geq Vol(Conv(X)) - \epsilon
\end{align*}

As per the discussion at the beginning of this subsection, in practice we will condition on $\mu_0(\hat{x})$ instead of $\int_0^{\infty} e^{-t} w_t(x_{\beta_t})^2 dt$. We will also use the quantity $|X|$ instead of $Vol(Conv(X))$. That is, if the set we are removing whose indices correspond to $P$ can be chosen such that $\mu_0(\hat{x}) \leq \frac{\epsilon}{|P|n|X|}$, then $Vol(Conv(X_{\bar{P}}))$ will be close to $Vol(Conv(X))$.

%Thus if $\mu_0(x_{\beta})$ can be chosen such that $\mu_0(x_{\beta}) \leq \frac{\epsilon}{|P|nVol(Conv(X))}$ for some $\epsilon > 0$ we have that
%
%\begin{align*}
%Vol(Conv(X)) \geq Vol(Conv(X_{\bar{P}})) \geq Vol(Conv(X)) - \epsilon
%\end{align*}

\section{Application of Zeroth Moment to Approximating Convex Hulls}\label{section_CH}

\subsection{Convex Hulls}

In this section we will describe, by employing the definition of moment, a simple filtering technique to approximate the convex hull of a collection of points in $\R^d$. We first give a runtime analysis of these algorithms, and then give the results of a few experiments approximating the convex hull of a collection of points.

\begin{algorithm}[H]\label{CHApprox}
\SetAlgoLined
%\KwResult{Remove points from set conditioned on zeroth moment }
 \KwIn{$X \subset \R^d$ finite set of $n$ points, $\epsilon \geq 0$ error threshold }\
 Calculate $\mu_0(x)$ for all $x \in X$\
 
Label the points of $X$ in ascending order of the values of $\mu_0(x)$; $\mu_0(x_0) \leq ... \leq \mu_0(x_n)$\

Search for the largest index $i$ such that $\mu_0(x_i) \leq \frac{\epsilon}{di|X|}$ \

\KwRet{$Conv(\{x_i,...,x_n\})$ }

% \While{While condition}{
%  instructions\;
%  \eIf{condition}{
%   instructions1\;
%   instructions2\;
%   }{
%   instructions3\;
%  }
% }
 \caption{Approximate Convex Hull}
\end{algorithm}

%    \begin{algorithm}[H]\label{CHApproxPll}
%    \SetAlgoLined
%    \KwIn{$X \subset \R^d$ finite set of $n$ points, $\epsilon \geq 0$ error threshold }\
%    % Calculate $\mu_0(x)$ for all $x \in X$\
%    $\hat{X} = set()$\
%    
%    \For{$x \in X$}{
%    	\uIf{$\mu_0(x) \geq \frac{\epsilon}{dn|X|}$}{
%    		$\hat{X}$.add($x$)\
%    		}
%    }\
%    \KwRet{$Conv(\hat{X})$ }
%    
%    % \While{While condition}{
%    %  instructions\;
%    %  \eIf{condition}{
%    %   instructions1\;
%    %   instructions2\;
%    %   }{
%    %   instructions3\;
%    %  }
%    % }
%    \caption{Approximate Convex Hull Parallel}
%    \end{algorithm}

\subsection{Runtime Analysis}

Steps $1\mhyphen3$ in Algorithm \ref{CHApprox} can be viewed as a preprocessing step to approximating the convex hull. The computation of the weight vector $w$ for a finite collection of points $X \subset \R^d$ can be parallelized (via parallel algorithms for matrix inversion, see e.g. \cite{MatrixComputations}, \cite{MtxInvPll}); thus can be minimized in accordance with available resources. Otherwise, matrix inversion has computational complexity $\OO(n^{\omega})$, where $\omega$ is matrix multiplication time. It should also be noted that $w$ can be computed by solving the linear system $\zeta_Xw = \one$ for $w$. Then sorting the values of $X$ by their moment can be done in $\OO(n\log n)$, and searching for the greatest index meeting the criteria in Algorithm \ref{CHApprox} step 3 can be done in $\OO(\log n)$ using a binary search.

%We also illustrate in Algorithm \ref{CHApproxPll} an algorithm that can be easily parallelized, save for the final step of computing the convex hull. It should be noted that computing the convex hull can be parallelized \cite{CHPll}, however, the parallelization method is not as na\"ive as parallelizing the prior steps. However, as we can see in line 3 of Algorithm \ref{CHApproxPll}, this comes at a cost of being more stringent with respect to accepting candidate points.

Although the convex hull of a finite set of points in $\R^d$ can be computed using the Quickhull algorithm \cite{Barber:1996:QAC:235815.235821}, whose average complexity is taken to be $\OO(n\log(n))$, the worst case can potentially have complexity $\OO(n^2)$. This indicates that there is still potential to suffer from extreme compute times while computing the convex hull if the size of the data set is large. Algorithms \ref{CHApprox} can potentially be used to reduce the number of data points being used to compute the convex hull, to mitigate the computation time.

%Steps $1\mhyphen3$ in Algorithm \ref{CHApprox} can be viewed as a preprocessing step to approximating the convex hull. . For this preprocessing step to be effective, it must not be costly. The computation of the weight vector $w$ for a finite collection of points $X \subset \R^d$ can be paralellized (via parallel algorithms for matrix inversion); thus can be minimized in accordance with available resources. Otherwise, matrix inversion has computational complexity $\OO(n^{\omega})$, where $\omega$ is matrix multiplication time. It should also be noted that $w$ can be computed by solving the linear system $\zeta_Xw = \one$ for $w$. Then sorting the values of $X$ by their moment can be done in $\OO(k\log k)$, and searching for the greatest index meeting the criteria in Algorithm \ref{CHApprox} step 4 can be done in $\OO(\log k)$ using a binary search. Thus the computational complexity of the convex hull computed in Algorithm \ref{CHApprox}, will be $\OO((k-i)^d)$, reduced from $\OO(k^d)$.

%      D. Stefanica, A Linear Algebra Primer for Financial Engineering. Fe
%     Press, 2014.
%
%     http://www.fepress.org/wp-content/uploads/2014/06/nla_primer-toc.pdf
%     theorem 7.6 pg 203
%
%P. Diaconis, S. Holmes, and M. Shahshahani. Sampling from a manifold, 2012. arXiv 1206.6913.

\subsection{Experiment}

In this section we describe numerical experiments measuring how fast the convex hull is recovered using Algorithm \ref{CHApprox}. Given a data set $X \subset \R^d$, the points can be ordered in descending order of their moment; that is, order $X = \{ x_1, ... , x_n \}$ such that $\mu_0(x_1) \geq \cdots \geq \mu_0(x_n)$. For a given $i$, write $X_{\leq i} = \{ x_1,...,x_i \}$, that is $X_{\leq i}$ are the $i$ points of $X$ with the highest moment. We look at $Vol(Conv(X_{\leq i}))$ as well as $|X_{\leq i}|$ for $i = 1, ... , n$ to determine to what extent the convex hull of $X$ is captured by the convex hull of points in $X$ with high moment. We also make special note of the smallest value of index $I$ such that $Vol(Conv(X_{\leq I}))$ is greater than $90\%$ of $Vol(Conv(X))$.

%For comparison, we also record the number of vertices of $Conv(X)$. 
In our experiments, we sampled $X$ from three Gaussian distributions; for an example, see D) in Figure \ref{weight_plots}. We generated 20 data sets in each of dimensions $2, 3, 4, 5$ and recorded the quantities discussed above. Table \ref{ch_approx_exp_results} summarizes the results of our experiments. Figure \ref{ch_approx_exp} shows the plots of $Vol(Conv(X_{\leq i}))$ (blue) and $|X_{\leq i}|$ (orange) for randomly selected data set for each of dimensions $2, 3, 4, 5$, as well as a horizontal line (black) marking when $Vol(Conv(X_{\leq i}))$ reaches at least $90\%$ of the total of $Vol(Conv(X))$.

\begin{figure}
  \includegraphics[scale=0.3]{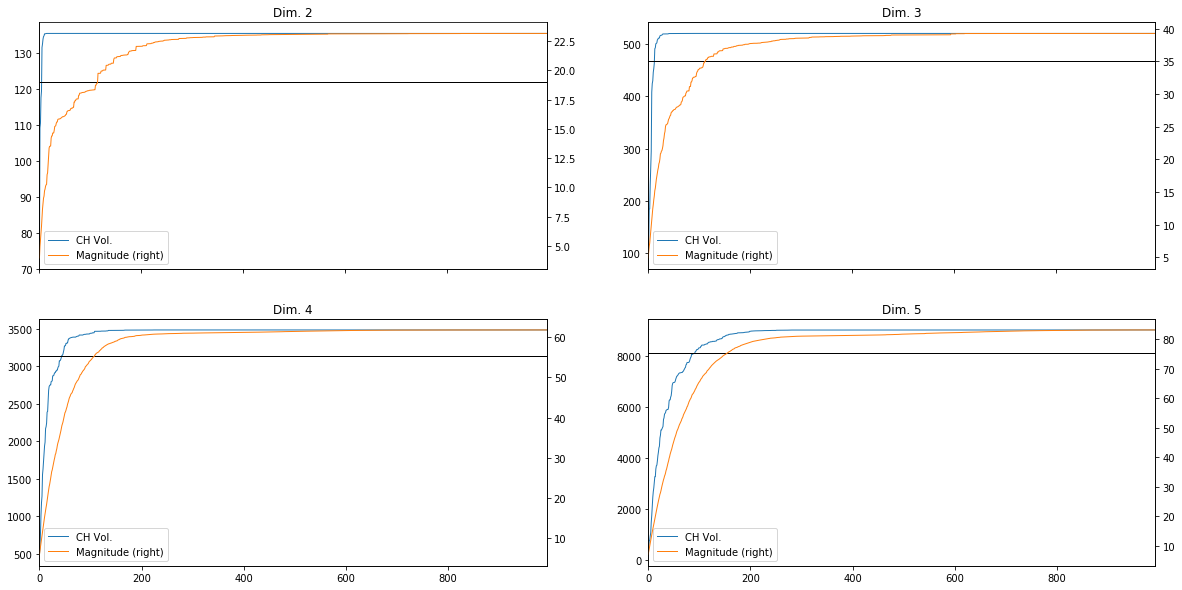}
  \caption{}
  \label{ch_approx_exp}
\end{figure}

\begin{table}[h]
\begin{center}
\begin{tabular}{ |m{4em}|m{4em}|m{4em}|m{4em}|m{4em}|m{4em}| } 
 \hline
 \textbf{Num. points} & \textbf{Dimension} & \textbf{Avg. \# points to 90\% CH volume} & \textbf{Std. dev. points to 90\% CH volume} & \textbf{Avg. \# CH vertices} & \textbf{Std. dev. \# CH vertices} \\ 
 \hline
1000 & 2 & 4.1 & 1.21 & 12.9 & 2.3 \\
\hline
1000 & 3 & 15.7 & 2.4 & 45.1 & 4.43 \\
\hline
1000 & 4 & 43.35 & 6.32 & 110 & 10.9\\
\hline
1000 & 5 & 80 & 10.2& 203 & 16 \\
 \hline
\end{tabular}
\caption{Summary statistics of experiments approximating convex hull}
\label{ch_approx_exp_results}
\end{center}
\end{table}

\section{Conclusion}

We have investigated in more detail the significance of the weight vector for a finite subset of $\R^n$, and introduced the notion of the moment of a point in order to give a measurement that carries important geometric information. We used the moment as an ordering for the points that is useful for selecting points when approximating the convex hull. Future directions of investigation include further exploring the connection between the weight and moment vectors of a finite set $X \subset \R^n$ and its geometric structure; the definition of moment \ref{moment} suggests that the quantities  
\[
\mu_n(x) =  \int_0^{\infty} t^n e^{-t}w_t(x)^2dt
\]

are interesting, as well as the (shifted) Laplace transform of $w_t(x)^2$:

\[
\mathcal{L}\{w_t\}(s) = \int_0^{\infty}  e^{-(s+1)t}w_t(x)^2dt
\]
Potential applications to computational geometry include dynamic convex hull computation, and range searching. Applications to machine learning are also being pursued.

%\begin{remark}
%Although not investigated in this paper, the definition above suggests that the quantities
%
%\[
%\mu_n(x) =  \int_0^{\infty} t^n e^{-t}w_t(x)^2dt
%\]
%
%are interesting, as well as the (shifted) Laplace transform of $w_t(x)^2$:
%
%\[
%\mathcal{L}\{w_t\}(s) = \int_0^{\infty}  e^{-(s+1)t}w_t(x)^2dt
%\]
%\end{remark}

\bibliographystyle{plain}
\bibliography{magnitude_convex_hull}

\begin{thebibliography}{1}

\bibitem{Barber:1996:QAC:235815.235821}
C.~Bradford Barber, David~P. Dobkin, David~P. Dobkin, and Hannu Huhdanpaa.
\newblock The quickhull algorithm for convex hulls.
\newblock {\em ACM Trans. Math. Softw.}, 22(4):469--483, December 1996.

\bibitem{Barcelo18CptEuclid}
{Juan Antonio} Barcelo and {Anthony} Carbery.
\newblock {On the magnitudes of compact sets in Euclidean spaces}.
\newblock {\em {American Journal of Mathematics}}, 140(2):449--494, 2018.

\bibitem{MatrixComputations}
{Gene}~H. Golub and {Charles} F.~Van Loan.
\newblock {\em {Matrix Computations}}.
\newblock Johns Hopkins University Press, third edition edition, Oct 1996.

\bibitem{MtxInvPll}
{Michael} Lass, {Stephan} Mohr, {Hendrik} Wiebeler, {Thomas}~D. K{\:u}hne, and
  Christian Plessl.
\newblock {A Massively Parallel Algorithm for the Approximate Calculation of
  Inverse p-th Roots of Large Sparse Matrices}.
\newblock {\em Proceedings of ACM Conference}, April 2017.

\bibitem{Leinster2013TheMO}
Tom Leinster.
\newblock The magnitude of metric spaces.
\newblock {\em Documenta Mathematica}, 18:857–905, 2013.

\bibitem{LeinsterMeckes16}
Tom Leinster and Mark Meckes.
\newblock The magnitude of a metric space: From category theory to geometric
  measure theory.
\newblock Aug 2018.

\bibitem{LeinsterWillerton2013}
Tom Leinster and Simon Willerton.
\newblock On the asymptotic magnitude of subsets of euclidean space.
\newblock {\em Geometriae Dedicata}, 164(1):287--310, Jun 2013.

\bibitem{MeckesPosDef}
{M. W.} Meckes.
\newblock Positive definite metric spaces.
\newblock {\em Positivity}, 17:733--757, Sept 2013.

\bibitem{Ruiz2016RelationshipBT}
E.~Ju{\'a}rez Ruiz, R.~Cortes Maldonado, and Francisco Rodr{\'i}guez.
\newblock Relationship between the inverses of a matrix and a submatrix.
\newblock {\em Computaci{\'o}n y Sistemas}, 20, 2016.

\end{thebibliography}

\end{document}